\documentclass[12pt,a4paper,twoside,reqno]{amsart}
\usepackage[a4paper,top=3cm,bottom=3cm,inner=2.5cm,outer=2.5cm]{geometry}
\usepackage[utf8]{inputenc}
\usepackage[british]{babel}
\usepackage[square]{natbib}
\usepackage{graphicx,xcolor}
\usepackage{xspace,centernot}
\usepackage[inline]{enumitem}
\usepackage{amsthm,amsfonts,amsmath,amssymb}
\usepackage[colorlinks=true]{hyperref}
\hypersetup{allcolors=[rgb]{0.1,0.1,0.4}}
\usepackage{tikz}
\usetikzlibrary{cd}
\tikzcdset{every label/.append style={font=\normalsize}}
\tikzset{iso/.style={draw=none,every to/.append style={edge node={node [sloped, allow upside down, auto=false]{$\cong$}}}}}
\usepackage[all]{xy}

\newdir{|>}{!/4.5pt/@{|}*:(1,-.2)@^{>}*:(1,+.2)@_{>}}
\mathcode`\:="603A     
\mathcode`\<="4268     
\mathcode`\>="5269     
\mathchardef\gt="313E  
\mathchardef\lt="313C  
\mathchardef\colon="303A  

\newtheorem{theorem}{Theorem}[section]

\newtheorem{proposition}[theorem]{Proposition}
\theoremstyle{definition}
\newtheorem{definition}[theorem]{Definition}

\newtheorem{remark}[theorem]{Remark}
\newtheorem{remarks}[theorem]{Remarks}

\newcommand\fct[1][1.4em]{\tikz[baseline=-0.5ex, 
	shorten <=2pt, shorten >=2pt] \draw[->] (0,0) -- (#1,0);}
\def\Epsilon{\mathrm{E}}
\def\nameit#1{#1~}
\def\thx{\nameit{Theorem}}

\def\prx{\nameit{Proposition}}

\def\dfx{\nameit{Definition}}

\def\rmx{\nameit{Remark}}

\def\eg{\emph{e.g.}\xspace} \def\ie{\emph{i.e.}\xspace}

\def\dfn#1{\emph{#1}}

\edef\cdrestoreat{
	\noexpand\catcode\lq\noexpand\@=\the\catcode\lq\@}\catcode\lq\@=11

\def\opfont#1{{\operator@font{#1}}}
\def\opr#1{\mathord{{\opfont{#1}}}}

\def\cco{\@ifnextchar[{\CCO}{\CCO[6ex]}}
\cdrestoreat
\def\CCO[#1]#2#3#4#5{\ensuremath{(\xymatrix@1@=#1{#2\ar[r]|-{#3}&#5},#4)}}
\def\M{\ensuremath{\mathbf{M}}\xspace}
\def\N{\ensuremath{\mathbf{N}}\xspace}

\DeclareFontFamily{U}{min}{}
\DeclareFontShape{U}{min}{m}{n}{<-> udmj30}{}
\def\yo{\text{\!\usefont{U}{min}{m}{n}\symbol{'210}\!}}

\DeclareFontFamily{OT1}{pzc}{}
\DeclareFontShape{OT1}{pzc}{m}{it}{<->s*[1.14]pzcmi7t}{}
\DeclareMathAlphabet{\mathpzc}{OT1}{pzc}{m}{it}

\def\Yo{\ensuremath{\mathop{\yo}}\xspace}
\def\ct#1{\ensuremath{\mathpzc{#1}}}
\def\Ct#1{\ensuremath{\mathbf{#1}}}
\def\id#1{\ensuremath{\mathrm{id}_{#1}}}
\def\Id#1{\ensuremath{\mathrm{Id}_{#1}}}
\def\op{^{\textrm{\scriptsize op}}}

\def\vuoto{}
\def\ple#1{\relax\def\teste{#1}\relax
	\ensuremath{\left\langle\relax
		\ifx\teste\vuoto\relax{\,}\else{#1}\fi\relax
		\right\rangle}}

\def\pnt{\raisebox{.2ex}{\scalebox{1.25}{$\textstyle\int$}}\kern-.1ex}

\def\blank{\textrm{--}}
\def\xycar{\ar@{--)}}
\def\xycocar{\ar@{--+}}
\def\ntar{\ar|>{}|-*=0[@]{{\raisebox{.7ex}{\scalebox{.3}{$\bullet$}}}}}
\def\ntarb{\ar|>{}|-*=0[@]{{\raisebox{-1.1ex}{\scalebox{.3}{$\bullet$}}}}}

\def\ocn#1#2{\ensuremath{%
		\relax\ifcase#1\relax\or{\widehat{#2}}
		\else{\left(\widehat{#2}\right)}\fi}}

\def\isar{\ar|>{}|-*=0[@]{\widetilde{\kern1ex}}}
\def\isarb{\ar|>{}|-*=0[@]{\widetilde{\kern1ex}}}
\def\fuar{\ar@<-.2ex>}
\def\onar{\ar@{-{|>}}}
\def\oaar{\ar}
\def\moar{\ar@{>->}}
\def\emar{\ar@<-.2ex>@{^(->}}
\def\taar{\ar@{=>}}
\def\ntto{\xymatrix@1@C=1.3em{{}\ntar[r]&{}}}
\def\isto{\xymatrix@1@C=1.3em{{}\isar[r]&{}}}
\def\futo{\xymatrix@1@C=1.6em{{}\fuar[r]&{}}}
\def\onto{\xymatrix@1@C=1.3em{{}\onar[r]&{}}}
\def\oato{\xymatrix@1@C=1.3em{{}\oaar[r]&{}}}
\def\moto{\xymatrix@1@C=1.3em{{}\,\moar[r]&{}}}
\def\emto{\xymatrix@1@C=1.6em{{}\emar[r]&{}}}
\def\tato{\xymatrix@1@C=1.3em{{}\taar[r]&{}}}
\def\larr{\mathrel{\xymatrix@1@=3.5ex{*=0{}\ar[];[r]&*=0{}}}}
\def\to{\mathrel{\xymatrix@1@=2.5ex{*=0{}\ar[];[r]&*=0{}}}}

\begin{document}

\title{A comonad for Grothendieck fibrations}
\author[Emmenegger]{Jacopo~Emmenegger}
\address{DIMA, Universit\`a di Genova, via Dodecaneso
	35, 16146 Genova, Italy}
\email{emmenegger@dima.unige.it}
\author[Mesiti]{Luca~Mesiti}
\address{School of Mathematics, University of Leeds, Leeds
	LS2 9JT, United Kingdom}
\email{mmlme@leeds.ac.uk}
\author[Rosolini]{Giuseppe~Rosolini}
\address{DIMA, Universit\`a di Genova, via Dodecaneso 35,
	16146 Genova, Italy}
\email{rosolini@unige.it}
\author[Streicher]{Thomas~Streicher}
\address{Fachbereich Mathematik, TU Darmstadt,
	Schlossgartenstra\ss e 7, 64289 Darmstadt, Germany}
\email{streicher@mathematik.tu-darmstadt.de}

\begin{abstract}
	We prove that cloven Grothendieck fibrations over a fixed base $\ct{B}$ are the pseudo-coalgebras for a lax idempotent 2-comonad on $\ct{Cat}/\ct{B}$. We show this via an original observation that the known colax idempotent 2-monad for fibrations over a fixed base has a right 2-adjoint. As an important consequence, we obtain an original cofree construction of a fibration on a functor. We also give a new, conceptual proof of the fact that the forgetful 2-functor from split fibrations to cloven fibrations over a fixed base has both a left 2-adjoint and a right 2-adjoint, in terms of coherence phenomena of strictification of pseudo-(co)algebras. The 2-monad for fibrations yields the left splitting and the 2-comonad yields the right splitting. Moreover, we show that the constructions induced by these coherence theorems recover Giraud's explicit constructions of the left and the right splittings.
\end{abstract}

\maketitle

\section{Introduction}

{Grothendieck introduced fibrations in the late 1950s for the purpose of studying descent
	in algebraic geometry. Fibred categories, as fibrations are also called, were then studied,
	among others, by B\'enabou, \citet{GiraudJ:cohna}, \citet{GrayJ:fibacc}, and
	\citet{StreetR:fibayl}.
	B\'enabou developed the theory in the 1970s for the purpose of doing category theory
	over a general base topos and even more generally over categories with finite limits. 
	Unfortunately, most of the material is unpublished. Copies of original manuscripts are
	available online at \href{https://www2.mathematik.tu-darmstadt.de/~streicher/FibCatTexts}
	{{\tt www2.mathematik.tu-darmstadt.de/\~{ }streicher/FibCatTexts}},
	see also the notes by \citet{StreicherT:fibc}.
	Marta Bunge was a strong advocate for his approach, see \eg
	\citep{BungeM:stacme,BungeM:staeic,BungeM:psetsc}.}

{An important well-known result is that cloven fibrations over a fixed base $\ct{B}$ are
	the pseudo-algebras for a colax idempotent 2-monad \M on $\ct{Cat}/\ct{B}$. Strict
	algebras for such 2-monad correspond to split fibrations. (Co)lax idempotent 2-monads were
	introduced by Kock in his PhD thesis~\citep{KockA:thesis}, with examples coming from
	colimit completion processes. Z\"oberlein then studied the corresponding concept for
	pseudo-monads in his PhD thesis~\citep{Zoberlein:doc2c}. The idea of a (co)lax idempotent
	2-monad is to encode a property-like structure, see also \citep{KockA:monfws}. Indeed
	objects of the base can have at most one structure of pseudo-algebra for such 2-monads.}

The free fibration on a functor was firstly constructed by Gray. In
\citep[Theorem~3.9]{GrayJ:fibacc}, he produced a left adjoint to the forgetful
functor from the category $\ct{SpFib}(\ct{B})$ of split fibrations over $\ct{B}$ and
functors that preserve the cleavage on the nose into the slice category
$\ct{Cat}/\ct{B}$. The monadicity of this forgetful functor has then been known to some
extent for many years. A complete proof appeared recently in Chapter 9 of Johnson and
Yau's book \citeyearpar{Johnson:2dimc}. The algebraic character of fibrations over
functors was further clarified by Street, 
who introduced fibrations over a fixed base internal to a (suitable) 2-category
as pseudo-algebras for a colax idempotent 2-monad, which generalizes the one for
Grothendieck fibrations, see~\citep[pp.~118,~122]{StreetR:fibayl}
where opfibrations and fibrations are respectively called 0-fibrations and 1-fibrations.

{One of the main results of the paper is that cloven fibrations over a fixed base $\ct{B}$
	are also the pseudo-coalgebras for a lax idempotent 2-comonad on $\ct{Cat}/\ct{B}$
	(\thx\ref{mthm}). We prove this by observing that the underlying 2-functor $M$ of the
	2-monad \M can be expressed as the composition of two left adjoint 2-functors. So that $M$
	has a right 2-adjoint $N$. This observation is original. By an extension of the classical
	argument of \citet[Proposition~3.3]{Eilenberg:adjft}, we can conclude that $N$
	underlies a lax idempotent 2-comonad $\N$ whose pseudo-coalgebras are isomorphic to the
	pseudo-algebras for $\M$, and thus coincide with cloven fibrations. We also give an
	explicit description of the 2-comonad $\N$ for fibrations in \rmx\ref{remdescr}, obtained
	via mating calculus.}

{The comonadicity result for Grothendieck fibrations has many important
	consequences. Notably, it provides in particular an original cofree construction of a
	fibration on a functor (\thx\ref{theorcofreefibration}). This gives a right adjoint to the
	forgetful functor from the category $\ct{SpFib}(\ct{B})$ of split fibrations over $\ct{B}$
	and functors that preserve the cleavage on the nose into the the slice category
	$\ct{Cat}/\ct{B}$. Such right adjoint was devised by the second author in his MSc thesis
	under the supervision of the third author. We could not find previous traces of it in the
	literature.}

{Another main result is a new, conceptual proof of the fact that the
	forgetful 2-functor from split fibrations to cloven fibrations over a fixed base has both
	a left 2-adjoint and a right 2-adjoint (\thx\ref{theortwosplittings}). We deduce this from
	the monadicity and the comonadicity results for Grothendieck fibrations. Both explicit
	left adjoint and right adjoint splittings of fibrations were introduced in
	\citep[Th\'eor\`emes~2.4.2 et~2.4.4]{GiraudJ:cohna}, which is unfortunately hard to
	read. The right splitting, with a proof that the counit consists of equivalences, then
	appeared in B\'enabou's lectures in 1974 at Montreal and in 1980 at Louvain-la-Neuve, with
	a construction based on his ``fibered Yoneda lemma''. The left splitting was brought to
	public attention again in \citep{KapulkinK:simmufav}. The two
	adjoints gained interest as they provide two ways of turning a fibration into an
	equivalent split one: a description of this can be found
	in~\citep[Section~3]{StreicherT:fibc}.}

We prove that both left and right splittings coincide with coherence phenomena of
strictification of pseudo-(co)algebras, studied by \citet{Power:gencr} and
\citet{Lack:codoc}. The 2-monad $\M$ yields the left adjoint splitting as the left 
adjoint to the forgetful from strict algebras to pseudo-algebras. The 2-comonad $\N$
yields the right adjoint splitting as the right adjoint to the forgetful from strict
coalgebras to pseudo-coalgebras. 

{Moreover, we show that the recipes described in \citep{Lack:codoc} and
	\citep{Power:gencr} to strictify pseudo-(co)algebras recover the
	explicit constructions of the right and left splittings given by \citet{GiraudJ:cohna}
	(\rmx\ref{remrightsplitting} and \rmx\ref{remleftsplitting}). We believe this sheds new
	light on Giraud's explicit constructions.}

\subsection*{Outline of the paper}

In Section~\ref{sectionmonad}, we recall that cloven fibrations over a fixed base $\ct{B}$
are the pseudo-algebras for a colax idempotent 2-monad \M on $\ct{Cat}/\ct{B}$.   

In Section~\ref{sectioncomonad}, we prove that cloven fibrations over $\ct{B}$ are the
pseudo-coalgebras for a lax idempotent 2-comonad \N on $\ct{Cat}/\ct{B}$. We show this via
an original observation that the 2-monad \M has a right 2-adjoint. 

{In Section~\ref{sectioncofreefibration}, we show that the comonadicity result of
	Section~\ref{sectioncomonad} provides in particular an original cofree construction of a
	fibration on a functor.}

In Section~\ref{sectionsplittings}, we give a new, conceptual proof of the two splittings
of fibrations, in terms of coherence phenomena of strictification of
pseudo-(co)algebras. We also show that the constructions induced by these coherence
theorems recover Giraud's explicit construction of the left and right splitting. 

\subsection*{Acknowledgements}

We would like to acknowledge the two reviewers' very helpful suggestions. They were
extremely useful to redraft our original submission as well as to complement it with
references of which we were originally unaware.

\section{A colax idempotent monad for fibrations}\label{sectionmonad}

In this section, we recall that Grothendieck fibrations over a fixed base $\ct{B}$ are the
pseudo-algebras for a colax idempotent 2-monad on $\ct{Cat}/\ct{B}$. The result has been
known to some extent for many years. \citet[Theorem~3.9(ii)]{GrayJ:fibacc}
constructed a left adjoint to the inclusion into $\ct{Cat}/\ct{B}$ of its sub-category
$\ct{SpFib}(\ct{B})$ on split Grothendieck fibrations over \ct{B} and functors that
preserve the cleavage on the nose. It is quite straightforward to see that the canonical
comparison functor into the (strict) algebras for the monad generated by Gray's adjunction
is an isomorphism. By taking the 2-cells in $\ct{SpFib}(\ct{B})$ to be all 2-cells in
$\ct{Cat}/\ct{B}$, this upgrades to a 2-isomorphism between the 2-category of (strict)
algebras and (strict) algebra morphisms and the 2-category $\ct{SpFib}(\ct{B})$. One can
then see that the 2-category of pseudo-algebras is isomorphic to the 2-category
$\ct{Fib}(\ct{B})$ of cloven fibrations and functors that preserve the cleavage up to
isomorphism. 

{The monadicity result for fibrations played a crucial role in \citep{StreetR:fibayl},
	where fibrations in a 2-category are introduced precisely as the pseudo-algebras for a
	colax idempotent 2-monad which generalizes the one for Grothendieck fibrations. Recently,
	a complete proof of the monadicity result for fibrations appeared in Chapter 9 of
	\citep{Johnson:2dimc}.}

(Co)lax idempotent 2-monads were introduced by Kock in his PhD thesis~\citep{KockA:thesis},
with examples coming from colimit completion processes. Z\"oberlein then studied the
corresponding concept for pseudo-monads in his PhD thesis~\citep{Zoberlein:doc2c}. Another
useful reference is \citep{PowerA:reprfc}. 

\begin{definition}[Kock, Z\"oberlein]\label{defKZ}
	A 2-monad $\M=(M,\mu,\eta)$ on a 2-category $\ct{K}$ is called \dfn{colax idempotent}
	$($or \dfn{coKZ}$)$ if either of the following equivalent conditions holds: 
	\begin{enumerate}[label=\rm{(\roman*)}]
		\item the multiplication $\mu$ is right adjoint left inverse of $\eta M$; 
		\item any pseudo-algebra structure map $\alpha:M(K)\to K$ for $M$ is right adjoint left
		inverse of the unit $\eta_K$. 
	\end{enumerate}
\end{definition}

Notice that for a colax idempotent 2-monad $\M$ on $\ct{K}$, any object of $\ct{K}$ can
have at most one structure of pseudo-algebra for $\M$ up to isomorphism. In this sense, a
colax idempotent 2-monad encodes a property-like structure. 

\begin{remark}
	We would like to produce a free Grothendieck fibration on a functor
	$F:\ct{A}\futo\ct{B}$. So, for every object $a$ in \ct{A} and map $f:b\to F(a)$, we want
	to force 
	the existence of a cartesian lifting of $f$ to $a$. The idea is then to freely add all
	pairs $\cco{b}{f}{a}{F(a)}$, and thus to consider the comma category
	$\ct{B}/F$.\footnote{We shall often write an object in a comma category $\ct{B}/F$ as a
		pair \cco bfa{F(a)} instead of the appropriate triple
		$(b,a,\xymatrix@1@=6ex{b\ar[r]|-{f}&F(a)})$ when no confusion arises.} This can be
	compared for example to the construction of the free monoid on a set.
\end{remark}

\begin{proposition}[Gray, Street]\label{propdefmonad}
	The 2-functor $M:\ct{Cat}/\ct{B}\futo \ct{Cat}/\ct{B}$ that maps a functor
	$F:\ct{A}\futo\ct{B}$ into the functor on the left 
	in the comma object diagram in $\ct{Cat}$
	$$\xymatrix{\ct{B}/F\fuar[r]^{P_2}\fuar[d]_-{P_1}&\ct{A}\fuar[d]_{}="f"^-{F}\\
		\ct{B}\fuar[r]^{}="g"_-{\Id{\ct{B}}}&\ct{B}\ntar"g";"f"}$$
	extends to a colax idempotent 2-monad $\M=(M,\mu,\eta)$.
\end{proposition}
\begin{proof}
	By the universal property of the comma object, the identity natural transformation induces the unit $\eta_F:\ct{A}\futo \ct{B}/F$ over $\ct{B}$. Explicitly,
	$$\eta_F(a)=\cco{F(a)}{\id{}}{a}{F(a)}.$$
	Then the pasting
	$$\xymatrix{\ct{B}/{P_1}\fuar[r]^{P'_2}\fuar[d]_-{P'_1}&\ct{B}/F\fuar[r]^{P_2}\fuar[d]_{}="w"^-{P_1}&\ct{A}\fuar[d]_{}="f"^-{F}\\
		\ct{B}\fuar[r]^{}="q"_-{\Id{\ct{B}}}&
		\ct{B}\fuar[r]^{}="g"_-{\Id{\ct{B}}}\ntar"q";"w".&\ct{B}\ntar"g";"f".}$$
	induces the multiplication $\mu_F:\ct{B}/{P_1}\futo \ct{B}/F$ over $\ct{B}$. Explicitly,
	$$\mu_F(\cco{b_1}{f_1}{\cco{b_2}{f_2}{a}{F(a)}}{b_2})=\cco[8ex]{b_1}{f_2f_1}{a}{F(a)}$$
	and it acts on the arrows by pasting the two commutative squares.
	
	Using the universal property of the comma object, it is straightforward to prove that $(M,\mu,\eta)$ is a colax idempotent 2-monad. 
	
	Explicitly, $\eta M\dashv \mu$ with the unit of the adjunction being the identity and counit with component $u_F:\eta_{M(F)}\mu_F\ntto \Id{M(MF)}$ the natural transformation given by the family of arrows in $\ct{B}/M(MF)$,
	which component at the index \cco{b_1}{f_1}{\cco{b_2}{f_2}{a}{F(a)}}{b_2} is
	$$\xymatrix{
		b_1\ar[d]_-{\id{b_1}}\ar[r]^-{\id{b_1}}&b_1\ar[d]_-{f_1}&
		&b_1\ar[d]^-{f_1}\ar[r]^-{f_2 f_1}&F(a)\ar[d]^-{\id{F(a)}}\\
		b_1\ar[r]_-{f_1}&b_2&
		&b_2\ar[r]_-{f_2}&F(a)}$$
	from \cco{b_1}{\id{}}{\cco[8ex]{b_1}{f_2 f_1}{a}{F(a)}}{F(a)} to \cco{b_1}{f_1}{\cco{b_2}{f_2}{a}{F(a)}}{b_2}.
\end{proof}

\begin{remark}
	Clearly the functor underlying the 2-monad $\M$ is polynomial, see\linebreak[4] \citep{GambinoN:polfap}.
\end{remark}

\begin{remark}
	The 2-monad $\M$ extends to a colax idempotent 2-monad on the 2-category $\ct{Cat}^2$ of
	arrows of $\ct{Cat}$, that applies $\ct{Cat}/B$ into itself and commutes with
	change-of-base functors. However, such extension does not seem to work well with the rest
	of the theory. See for example \rmx\ref{remstreetfibrations} and
	\rmx\ref{essentialtorestrictoverbase}.
\end{remark}

The following result follows from the theory of fibrations internal to a 2-category as
developed by \citet{StreetR:fibayl,StreetR:fibbic}.
The argument we present is based on the well-known fact that
fibrations can be characterized as those functors $F$
such that the unit of \M on $F$ has a right adjoint in $\ct{Cat}/B$,
see for instance~\citep[Theorem~2.7]{Weber:yostr2top}.
This is an instance of the general phenomenon observed by \citet{KockA:monfws},
with the caveat that Kock only considers left adjoint right inverses (\emph{lari\/}) to
the unit since he is working with \emph{lax}-idempotent monads (hence the left instead of
right), and \emph{normal} pseudo-algebras (hence the right inverse).
About the latter observation,
see also the discussion in \citep[p.~120]{StreetR:fibayl} following the proof of
Proposition~9.

{As we already recalled at the beginning of the section, \citet{Johnson:2dimc} devote
	Chapter~9 to a detailed proof of the following theorem.}

\begin{theorem}[Street, Johnson--Yau]\label{cps}
	The $2$-category of pseudo-algebras for the colax idempotent 2-monad $\M$ on
	$\ct{Cat}/\ct{B}$ is isomorphic to the $2$-category $\ct{Fib}(\ct{B})$ of cloven
	Grothendieck fibrations over \ct{B} and functors that preserve the cleavage up to
	isomorphism.
	
	The 2-category of strict algebras for $\M$ is isomorphic to the 2-category
	$\ct{SpFib}(\ct{B})$ of split Grothendieck fibrations over \ct{B} and functors that
	preserve the cleavage on the nose. 
\end{theorem}
\begin{proof}
	By the characterization theorem for pseudo-algebras for a colax idempotent monad of
	\citep{KockA:monfws} (see also \dfx\ref{defKZ}), a functor $F:\ct{A}\futo\ct{B}$
	sustains a structure of pseudo-algebra for $\M$ if and only if the unit 
	$$\xymatrix{\ct{A}\ar[r]^-{\eta_F}\ar[d]_-{F}&\ct{B}/F\ar[d]^-{M(F)}\\
		\ct{B}\ar[r]_-{\Id{\ct{B}}}&\ct{B}
	}$$
	has a right adjoint $\alpha$ in $\ct{Cat}/\ct{B}$.\footnote{Indeed,
		\citet{KockA:monfws} shows just a bijection between the objects.}
	For every $a$ in \ct{A} and $f:b\to F(a)$ in $\ct{B}$, the object
	$\alpha\cco{b}{f}{F(a)}{a}$ is over $b$ and the component of the counit of $\eta_F\dashv
	\alpha$ on $\cco{b}{f}{F(a)}{a}$ is a lifting of $f$ to $a$, since the counit is
	vertical. The couniversality of the counit precisely translates as such lifting being
	cartesian. 
	
	The strict axioms for strict algebras translates as the cloven fibration being split.
\end{proof}

\begin{remarks}\label{remstreetfibrations}
	(a) Recall from \citep{StreetR:fibbic} that a functor $F :\ct{A} \to \ct{B}$ is a
	\dfn{Street fibration} if, for every object $a$ in $\ct{A}$ and arrow $f:b \to F(a)$,
	there are an isomorphism $h_a:b_a \to b$
	and a cartesian arrow into $a$ over the composite $fh_a : b_a \to F(a)$. Street fibrations
	can be characterized as pseudo-algebras for the ``same'' monad \M of Grothendieck
	fibrations, but lifted to the pseudo-fibre $\ct{Cat}/\!/\ct{B}$, see
	\citep{StreetR:fibbic}. 
	The 2-category $\ct{Cat}/\!/\ct{B}$ has the same objects as $\ct{Cat}/\ct{B}$,
	1-cells are triangles filled with a natural isomorphism,
	and 2-cells are those 2-cells in $\ct{Cat}/\ct{B}$ which commute with the two natural
	isomorphisms. 
	For the characterization of Street fibrations,
	it is enough to observe that the underlying functor of \M lifts,
	and that unit and counit are still natural.
	The rest of the proof works the same as for \prx\ref{cps}.
	
	\noindent(b)
	As \M is a monad also on the whole 2-category $\ct{Cat}^2$,
	we can also look at the pseudo-algebras there. It is easy to see that the strict algebras
	are again split Grothendieck fibrations. Following \citet{KockA:monfws}, 
	pseudo-algebras for \M on $\ct{Cat}^2$ can also be characterized as those functors $F$
	such that the unit $(\Id{},\eta_F)$ has a right adjoint in $\ct{Cat}^2$. Let
	$$\xymatrix{\ct{B}/F\ar[r]^-{A}\ar[d]_-{F}&\ct{A}\ar[d]^-{M(F)}\\
		\ct{B}\ar[r]^-{B}&\ct{B}}$$
	be such a right adjoint.
	In particular,
	the functor $B$ is isomorphic to the identity on $\ct{B}$ via an isomorphism $\zeta:B\to\Id{\ct{B}}$.
	Unfolding the other conditions
	one sees that for every object $a$ in \ct{A} and arrow $f:b \to F(a)$,
	there is a cartesian arrow into $a$ over $f\zeta_b:B(b) \to F(a)$.
	It follows that a pseudo-algebra for \M in $\ct{Cat}^2$ is a Street fibration.
	However, the converse does not seem to hold.
	Indeed, in both cases, given $a$ in \ct{A} and $f:b\to F(a)$,
	we can choose an isomorphism $h$ such that the composite $fh$ has a cartesian lift to $a$.
	But in a Street fibration the choice of $h$ depends on the object $a$,
	whereas in a pseudo-algebra in $\ct{Cat}^2$ the choice is given uniformly for every object $a$ by the isomorphism $\zeta_b$.
	
	Notice also that a normal pseudo-algebra in $\ct{Cat}^2$, being a right adjoint right inverse to the unit,
	is in particular an arrow in $\ct{Cat}/\ct{B}$ and, in fact, a  right adjoint right inverse to the unit in $\ct{Cat}/\ct{B}$.
	It follows that a normal pseudo-algebra in $\ct{Cat}^2$ is a normal Grothendieck fibration.
	This fact suggests that (normal) Street fibrations are strictly more general than pseudo-algebras in $\ct{Cat}^2$.
\end{remarks}

\section{A lax idempotent comonad for fibrations}\label{sectioncomonad}

In this section, we prove an original result of comonadicity for Grothendieck
fibrations. More precisely, we prove that fibrations over a fixed base $\ct{B}$ are also
the pseudo-coalgebras for a lax idempotent 2-comonad on $\ct{Cat}/B$. We show this by an
original observation that the 2-monad $\M$ has a right 2-adjoint. 

We will see in the following sections that the comonadicity theorem has important
consequences. It provides in particular an original cofree construction of a fibration on
a functor (\thx\ref{theorcofreefibration}). Moreover the monad and the comonad induce
respectively the left splitting and the right splitting of fibrations
(\thx\ref{theortwosplittings}, \rmx\ref{remrightsplitting},
\rmx\ref{remleftsplitting}). The comonadicity of fibrations has also consequences for
(higher) elementary topos theory. 

To reach the comonadicity result for Grothendieck fibrations, we notice the following
useful equivalent description of the 2-monad. 

\begin{proposition}\label{propcomp}
	The 2-functor underlying the 2-monad $\M$ coincides with the composition
	$$\xymatrix@C=5em{\ct{Cat}/\ct{B}\fuar[r]^-{\opr{cod}^*}&
		\ct{Cat}/\ct{B}^{\ct{2}}\fuar[r]^-{\opr{dom}_\bullet}&\ct{Cat}/\ct{B}}$$
	of the 2-functor $\opr{cod}^*$ that calculates pullbacks along $\opr{cod}:\ct{B}^2\futo \ct{B}$ and the 2-functor $\opr{dom}_\bullet$ of postcomposition with $\opr{dom}:\ct{B}^2\futo \ct{B}$.
\end{proposition}
\begin{proof}
	The comma object in $\ct{Cat}$
	$$\xymatrix{\ct{B}/F\fuar[r]^{P_1}\fuar[d]_-{P_2}& \ct{B} \fuar[d]_{}="g"^-{\Id{\ct{B}}}\\
		\ct{A}\fuar[r]^{}="f"_-{F}&\ct{B}\ntarb"g";"f".}$$
	from $\Id{\ct{B}}$ to $F$ is equivalent to the pullback of $F$ along the lax limit of the
	arrow $\Id{\ct{B}}$ (that acts as a replacement):
	$$\xymatrix{
		\ct{B}/{F} \fuar[r]^{}\fuar[d]_-{P_2} &\ct{B}^2\fuar[r]^{\opr{dom}}\fuar[d]_{}="w"_-{\opr{cod}}&\ct{B}\fuar[d]_{}="f"^-{\Id{\ct{B}}}\\
		\ct{A}\fuar[r]^{}="q"_-{F}&
		\ct{B}\fuar[r]^{}="g"_-{\Id{\ct{B}}}&\ct{B}\ntarb"f";"g".}$$
\end{proof}

\begin{remark}
	The 2-monad $\M$ can thus be expressed as the composition of left adjoint 2-functors. This
	observation is original.  
	
	The adjunction $\opr{dom}_\bullet\dashv\opr{dom}^*$ is the usual change of base.
	And the functor $\opr{cod}^*$ has indeed a right 2-adjoint $\opr{cod}_*$ because
	$\opr{cod}$ is an opfibration. Hence it is 2-exponentiable as stated in
	\citet[Th\'eor\`eme~4.4]{GiraudJ:metdld} and 
	in \citet[2nd Proposition on p.~894]{ConducheF:sujdea},  
	but for an explicit proof of the characterisation of exponentiable functors, see
	\citet{StreetR:powf} where such functors are called \emph{powerful}. See also
	\citet[Theorem~2.16]{StreetR:comfat} for the extension to the case of internal categories
	in a cartesian closed category with pullbacks.
\end{remark}

\begin{proposition}\label{propmhasrightadj}
	The 2-functor underlying the 2-monad $\M$ has a right 2-adjoint $N$, expressed as the
	composite 
	$$\xymatrix@C=5em{\ct{Cat}/\ct{B}\fuar[r]^-{\opr{dom}^*}&
		\ct{Cat}/\ct{B}^{\ct{2}}\fuar[r]^-{\opr{cod}_*}&\ct{Cat}/\ct{B}}$$
	Therefore $N$ underlies a lax idempotent 2-comonad $\N=(N,\mu',\eta')$ on
	$\ct{Cat}/\ct{B}$.
\end{proposition}
\begin{proof}
	The mating calculus ensures that the double category of left adjoints in a
	2-category is isomorphic to the double category of right adjoints, see
	\citep{Kelly:reveltwocat}. It follows that the right 2-adjoint of a colax idempotent
	2-monad underlies a lax idempotent 2-comonad. 
\end{proof}

{For 1-dimensional monads, \citet{Eilenberg:adjft} showed that the coalgebras for the
	right adjoint of a monad coincide with the algebras for the monad. \citet{LaudaA:frobalg}
	then proved that this works for pseudo-monads as well, obtaining a 2-equivalence between
	the pseudo-coalgebras and the pseudo-algebras. In the present case, we can prove the
	following stricter result.}

\begin{theorem}\label{mthm}
	The $2$-category of pseudo-coalgebras for the lax idempotent 2-comonad $\N$ on
	$\ct{Cat}/\ct{B}$ is isomorphic to the $2$-category $\ct{Fib}(\ct{B})$ of cloven
	Grothendieck fibrations over $\ct{B}$ and functors that preserve the cleavage up to
	isomorphism.  
	
	The 2-category of strict coalgebras for $\N$ is isomorphic to the 2-category
	$\ct{SpFib}(\ct{B})$ of split Grothendieck fibrations over $\ct{B}$ and functors that
	preserve the cleavage on the nose. 
\end{theorem}
\begin{proof}
	The isomorphism between the double category of left adjoints in a
	2-category and the double category of right adjoints \citep{Kelly:reveltwocat} (via mating
	calculus) transforms the 2-category of pseudo-algebras for 
	\M into the the 2-category of pseudo-coalgebras for \N. Indeed, \citet{LaudaA:frobalg} showed
	that a right pseudo-adjoint to a pseudo-monad in a Gray-category
	is a pseudo-comonad with pseudo-coalgebras 2-equivalent to the pseudo-algebras of the
	pseudo-monad. We notice that, when the pseudo-adjunction is a strict 2-adjunction and the
	interchange rule is strict, the 2-equivalence becomes a 2-isomorphism. We can then
	conclude by \thx\ref{cps}. 
\end{proof}

\begin{remark}\label{essentialtorestrictoverbase}
	For the comonadicity of fibrations, it is essential to restrict to fibrations over a fixed base $\ct{B}$. The factorization of \prx\ref{propcomp} indeed only holds for the $2$-monad restricted to $\ct{Cat}/\ct{B}$.
\end{remark}

\begin{remark}\label{remdescr}
	We can calculate the 2-comonad $\N=(\opr{cod}_*\opr{dom}^*,\mu',\eta')$ more
	explicitly. Indeed, as observed by \citet{GiraudJ:metdld}, the isomorphism
	between homsets that gives the adjunction $\opr{cod}^*\dashv\opr{cod}_*$ 
	determines the explicit definition of $\opr{cod}_*$, up to isomorphism. Notice that
	$\opr{dom}^*$ sends every functor $F:\ct{A}\futo\ct{B}$ to the forgetful $F/\ct{B}\futo
	\ct{B}^{\ct{2}}$. Since the fibre of $\opr{cod}$ over $b$ in \ct{B} is $\ct{B}/b$, we find
	that $N(F)$ is the projection on the first component $\ct{G}_F\futo \ct{B}$, where
	$\ct{G}_F$ is the category defined as follows. Its objects are pairs \ple{b,X} where $b$
	is an object in \ct{B} and 
	$X$ is a functor such that
	$$\begin{tikzcd}
		\ct{B}/b\ar{rr}{X}\ar{rd}[']{\partial_0}&&\ct{A}\ar{ld}{F}\\
		&\ct{B}
	\end{tikzcd}$$
	commutes. An arrow \begin{tikzcd}[column sep=2em]
		\ple{f,\alpha}:\ple{b,X}\ar{r}&\ple{b',X'}\end{tikzcd}
	consists of an arrow \begin{tikzcd}[column sep=2em]f:b\ar{r}&b'\end{tikzcd} in \ct{B}
	and a natural transformation
	$$\begin{tikzcd}
		\ct{B}/b\ar{rr}{X}[name=a,']{}\ar{rd}[']{f\circ_{\ct{B}}\blank}&&\ct{A}\\
		&\ct{B}/b'\ar{ru}[name=b]{}[']{X'}&
		\ar["\scalebox{.3}{$\bullet$}"{sloped,anchor=south},from=a,to=b]{}[']{\alpha}
	\end{tikzcd}$$
	with vertical components.
	
	\noindent $N$ then sends every morphism $H:F\to F'$ in $\ct{Cat}/\ct{B}$ to
	$\ple{\Id{},H\circ -}:\ct{G}_F\futo\ct{G}_{F'}$, which is over $\ct{B}$. And the action of
	$N$ on 2-cells is analogous. 
	
	By the mating calculus, the counit $\eta'$ of the 2-comonad $\N$ is the composite
	$$\opr{cod}_*\opr{dom}^*\xrightarrow[\eta\hspace{0.2ex} \opr{cod}_*\opr{dom}^*]{\scalebox{.3}{$\bullet$}}\opr{dom}_\bullet \opr{cod}^* \opr{cod}_*\opr{dom}^*\xrightarrow[\opr{dom}_\bullet \hspace{0.2ex} \xi \hspace{0.2ex} \opr{dom}^*]{\scalebox{.3}{$\bullet$}}\opr{dom}_\bullet \opr{dom}^*\xrightarrow[\zeta]{\scalebox{.3}{$\bullet$}}\id{},$$
	where $\xi$ is the counit of $\opr{cod}^*\dashv\opr{cod}_*$, whose components are evaluation functors, and $\zeta$ is the counit of $\opr{dom}_\bullet\dashv\opr{dom}^*$. Given a functor $F:\ct{A}\futo\ct{B}$, the component of $\eta'$ on $F$ is then given by ``evaluating at the identity''
	$$\begin{tikzcd}[row sep=.1ex,column sep=4em]
		\ct{G}_F\ar{r}{\Epsilon}&\ct{A}\\
		\ple{b,X}\ar{dd}[name=a]{\ple{f,\alpha}}\ar[mapsto]{r}&
		X(\id{b})\ar{dd}{X'(f:f\xrightarrow{}\id{b'})\circ\alpha_{\id{b}}}[name=b,']{}\\\strut\\
		\ple{b',X'}\ar[mapsto]{r}&X'(\id{b'})\ar[mapsto,from=a,to=b]\\
		X(\id{b})\ar{r}{\alpha_{\id{b}}}&X'(f)
	\end{tikzcd}$$
	
	The comultiplication $\mu'$ of the 2-comonad $\N$ has component $\mu'_F$ on $F:\ct{A}\futo \ct{B}$ given by the functor $\ct{G}_F\futo \ct{G}_{N(F)}$ that sends $\ple{b,X}$ to $\ple{b,\overline{X})}$ with
	$$\begin{tikzcd}[row sep=.1ex,column sep=4em]
		\ct{B}/b\ar{r}{\overline{X}}&\ct{G}_F\\
		(h:b'\rightarrow b)\ar[mapsto]{r}&
		\ple{b',X\circ(h\circ -)}
	\end{tikzcd}$$
	After \rmx\ref{remcoalgstructoffib}, we will be able to see $\mu'_F$ as the coalgebra structure map of the split fibration $N(F)$.
\end{remark}

{One could also use the isomorphism between homsets, given by the adjunction
	$M\dashv N$, as guaranteed by \prx\ref{propmhasrightadj}, to determine $N$ directly as
	follows.}

Given $b$ in \ct{B}, write $\Yo(b):{\ct{B}}/b\fct\ct{B}$
for the split fibration given by the (restriction of the) domain functor from the comma
category $\ct{B}/b$ to \ct{B}. That assignment extends to a functor
$$\begin{tikzcd}[row sep=.1ex,column sep=4em]
	\ct{B}\ar{r}{\Yo}&\opr{SpFib}(\ct{B})\\
	b_1\ar[mapsto]{r}\ar{dd}[name=f]{f}&
	\ct{B}/b_1\ar{dd}[name=g,']{f\circ_{\ct{B}}\blank}\\\strut \\
	b_2\ar[mapsto]{r}&\ct{B}/b_2\ar[mapsto,from=f,to=g]
\end{tikzcd}$$

\begin{proposition}\label{propcomonadisgrothconstr}
	Fix a functor $F:\ct{A}\fct\ct{B}$ in $\Ct{Cat}/\ct{B}$. Homming from each $\Yo(b)$ into
	$F$ in the comma 2-category $\Ct{Cat}/\ct{B}$
	$$\begin{tikzcd}[row sep=.1ex,column sep=4em]
		\ct{B}\op\ar{r}{\widehat{F}}&\Ct{Cat}\\
		b\ar{dd}[name=a]{f}[near end,']{\ct{B}\op}\ar[mapsto]{r}&
		(\Ct{Cat}/\ct{B})(\Yo(b),F)\ar{dd}[name=b,']{\blank\circ_{\Ct{Cat}/\ct{B}}\Yo(f)}\\\strut\\
		b'\ar[mapsto]{r}&(\Ct{Cat}/\ct{B})(\Yo(b'),F)\ar[mapsto,from=a,to=b]
	\end{tikzcd}$$
	gives a strict indexed category over $\ct{B}$, and its fibration of points
	$$\pnt\widehat{F}:\ct{G}_F\futo\ct{B}$$
	is precisely $N(F)$.
\end{proposition}

{Note that it follows directly from \prx\ref{propcomonadisgrothconstr} that $N(F)$ is
	split, which we know already from \thx\ref{mthm} as $N(F)$ is a cofree strict coalgebra.} 

\begin{remark}\label{remcoalgstructoffib}
	The pseudo-coalgebra structure map of a cloven fibration $p:\ct{E}\futo\ct{B}$ is,
	explicitly,
	$$\begin{tikzcd}
		\ct{E}\ar{rr}{\alpha}\ar{rd}[']{p}&&\ct{G}_F\ar{ld}[{inner sep=0.3ex}]{N(p)}\\
		&\ct{B}
	\end{tikzcd}$$
	where $\alpha$ sends $E$ in \ct{E} to the triangle
	$$\begin{tikzcd}[row sep=4.75ex,column sep=4.75ex]
		\ct{B}/{p(E)}\ar{rr}{(-)^* E}\ar{rd}[']{\partial_0}&&\ct{E}\ar{ld}{p}\\
		&\ct{B}
	\end{tikzcd}$$
	with $(-)^* E$ calculating the domains of the chosen cartesian liftings to $E$, and is
	extended by cartesianity to a functor $\alpha:\ct{E}\futo \ct{G}_F$. 
	
	The rest of the pseudo-coalgebra structure of $p$ is given by the isomorphisms
	$\alpha_{\eta'}$ and $\alpha_{\mu'}$ that regulate respectively liftings in the chosen
	cleavage of $p$ of an identity and of a composite. 
	
	{The well-known observation that the component at $F$ of the comultiplication $\mu'_F$ of
		the comonad $N$ is precisely the coalgebra structure map of the split fibration
		$N(F)$ translates, in the notation of \rmx\ref{remdescr}, as}
	$$\overline{X}=(-)^* \ple{b,X}.$$
	That is, $\overline{X}$ shows the explicit chosen cartesian liftings for the split
	fibration $N(F)$, which can be read from \prx\ref{propcomonadisgrothconstr}.
\end{remark}

\section{The cofree fibration on a functor}\label{sectioncofreefibration}

{An important consequence of the comonadicity result is the fact that the forgetful
	2-functor from the 2-category $\ct{SpFib}(\ct{B})$ of split fibrations over a fixed
	$\ct{B}$ to $\ct{Cat}/\ct{B}$ also has a right adjoint. This result seems not to appear in
	the literature.}

\begin{theorem}\label{theorcofreefibration}
	
	{The forgetful 2-functor $U:\ct{SpFib}(\ct{B})\futo \ct{Cat}/\ct{B}$ has a right
		2-adjoint. The cofree fibration on a functor $F:\ct{A}\futo \ct{B}$ is the split fibration
		$N(F):\ct{G}_F\futo\ct{B}$, described in \rmx\ref{remdescr} and
		\prx\ref{propcomonadisgrothconstr}.}
\end{theorem}

{\begin{proof}
		The right $2$-adjoint to $U$ is given by the forgetful-cofree adjunction 
		given by the 2-comonad $\N$. By \thx\ref{mthm}, the strict coalgebras for the 2-comonad
		$\N$ are precisely the split fibrations.
\end{proof}}

\begin{remark}
	
	{The fact that $N$ gives a right adjoint to $U$ could also be proved directly, showing
		that the counit $\eta'$ is 2-universal. The component on $p$ of the unit of the adjunction
		is given by the coalgebra structure map of the split fibration $p$, described in
		\rmx\ref{remcoalgstructoffib}. However a direct proof is considerably more involved than
		the one that just follows from the factorization of the 2-monad as the composite of two
		left adjoint 2-functors. It is indeed much easier to calculate the right adjoint to the
		monad \M than the right adjoint to $U$.}
\end{remark}

\begin{remark}
	
	{We briefly describe a strategy for a direct proof of the comonadicity of $U$.}
	
	{In order to prove that a coalgebra $(F:\ct{A}\futo \ct{B},\alpha:F\fct N(F))$ for the
		2-comonad $\N$ is a split fibration, start from a diagram}
	$$\begin{tikzcd}
		& a \arrow[d,mapsto,"{F}"]\\
		b \arrow[r,"{f}"'] & F(a)
	\end{tikzcd}$$
	{and construct a cartesian lifting of $f$ to $a$. The idea is to apply $\alpha$ to such
		a diagram and compute the cartesian lifting with respect to the split fibration $N(F)$ of
		$f$ to $\alpha(a)=\ple{F(a),X:\ct{B}/(F(a))\futo \ct{A}}$. Next, applying the counit
		$\eta'_F$ recovers the starting data of the diagram above and exhibits a lifting of $f$ to
		$a$:}
	$$\begin{tikzcd}[column sep=8ex]
		\eta'_F(\ple{b,X\circ (f\circ-)})
		\arrow[r,"{\eta'_F(\ple{f,\id{}})}"]\arrow[d,mapsto,"{F}"]& a
		\arrow[d,mapsto,"{F}"]\\ 
		b \arrow[r,"{f}"'] & F(a)
	\end{tikzcd}$$
	{Note that, by \rmx\ref{remdescr},
		$$\eta'_F(\ple{b,X\circ (f\circ -)})=X(f).$$
		One is left with proving directly that such a lifting of $f$ to $a$ is cartesian, whose
		proof is lengthy. The proof of \thx\ref{mthm} uses instead what was already known
		for the 2-monad $\M$.}
\end{remark}

\section{Recovering the two ways to split a fibration}\label{sectionsplittings}

In this section, we present a new, conceptual proof of the fact that the forgetful
2-functor from split fibrations to cloven fibrations over a fixed base has both a left
2-adjoint and a right 2-adjoint. Both explicit left and right splittings of fibrations
were introduced in \citep[I.2.4]{GiraudJ:cohna}.

{We prove that both left and right splittings coincide with coherence phenomena of
	strictification of pseudo-(co)algebras. We show that such strictification adjoints are
	guaranteed for the monad \M and the comonad \N by the theorems of \citet{Lack:codoc}. The
	2-monad $\M$ yields the left adjoint splitting as the left adjoint to the forgetful from
	strict algebras to pseudo-algebras. The 2-comonad $\N$ yields the right adjoint splitting
	as the right adjoint to the forgetful from strict coalgebras to pseudo-coalgebras.}

{Moreover, we show that the recipes described in \citep{Lack:codoc,Power:gencr} to
	strictify pseudo-(co)algebras concretely, recover the explicit constructions of right
	and left splittings given by \citet{GiraudJ:cohna}. This sheds new light on
	Giraud's explicit constructions.
	
	A conceptual proof of the two splittings of fibrations is obtained via a result of
	\citep[Theorem~3.2]{Lack:codoc}, which we shall apply in its dual form, as we recall in
	the following.} 

\begin{theorem}[Lack]\label{thm:cohcom}
	If $T$ is a 2-comonad on a 2-category $\ct{K}$ admitting descent objects, and $T$
	preserves them, 
	then the inclusion $T\text{-CoAlg}_s \futo \text{Ps-}T\text{-CoAlg}$ of strict coalgebras
	into pseudo-coalgebras has a right adjoint, 
	and the components of the counit are equivalences in $\text{Ps-}T\text{-Alg}$.
	In particular this is the case if $\ct{K}$ has iso-inserters and equifiers, and $T$
	preserves these.
\end{theorem}

The following two propositions show that the 2-monad $\M$ and the 2-comonad $\N$ satisfy
the assumptions of \thx\ref{thm:cohcom}, guaranteeing the existence of the
strictification adjoints. 

\begin{proposition}\label{propcatoverbhasdescentobj}
	$\ct{Cat}/\ct{B}$ has all weighted 2-colimits, created by the domain functor into
	$\ct{Cat}$. In particular, it has all codescent objects. 
	
	Moreover $\ct{Cat}/\ct{B}$ also has all iso-inserters and equifiers, thus all descent
	objects. 
\end{proposition}
\begin{proof}
	$\ct{Cat}/\ct{B}$ is the 2-category of strict coalgebras for the 2-comonad $-\times
	\ct{B}$ on $\ct{Cat}$. So the forgetful into $\ct{Cat}$ creates all weighted
	2-colimits. Since $\ct{Cat}$ is cocomplete as a 2-category, $\ct{Cat}/\ct{B}$ is
	cocomplete as well. As shown in \citep{Lack:codoc}, codescent objects are
	particular weighted 2-colimits. 
	
	The iso-inserter of two functors $F,G:\ct{C}\futo\ct{D}$ over $\ct{B}$ is given by the
	category whose objects are all pairs ${(C,\phi_C)}$ with $C$ in \ct{C} and
	$\phi_C:F(C)\cong G(C)$ an isomorphism in $\ct{D}$ over the identity, and whose morphisms
	${(C,\phi_C)}\to{(C',\phi_{C'})}$ are morphisms $f:C\to C'$ in $\ct{C}$ such that
	$G(f)\phi_C=\phi_{C'}F(f)$. We are thus restricting the usual inserter in $\ct{Cat}$
	taking only those $\phi_C$ that are vertical.
	
	Equifiers in $\ct{Cat}/\ct{B}$ are just calculated in $\ct{Cat}$. By
	\citep{Lack:codoc}, $\ct{Cat}/\ct{B}$ has then descent objects as well, as they
	can be produced via an iso-inserter followed by two equifiers. 
\end{proof}

\begin{proposition}\label{propmonadpreservescodescentobj}
	The 2-monad $\M$ preserves codescent objects. The 2-comonad $\N$ preserves descent
	objects.
\end{proposition}
\begin{proof}
	By \prx\ref{propmhasrightadj}, $M\dashv N$. So $M$ preserves all weighted 2-colimits and
	$N$ preserves all weighted 2-limits. 
\end{proof}

We can now present a new, conceptual proof for the following theorem, which firstly
appeared in \citep{GiraudJ:cohna}.

\begin{theorem}\label{theortwosplittings}
	The forgetful 2-functor $V:\ct{SpFib}(\ct{B})\futo \ct{Fib}(\ct{B})$ from split fibrations
	over $\ct{B}$ to cloven fibrations over $\ct{B}$ has both a left 2-adjoint $L$ and a right
	2-adjoint $R$. 
	
	Moreover, the components of the unit of the adjunction $L\dashv V$ are equivalences in
	$\ct{Fib}(\ct{B})$. And the components of the counit of $V\dashv R$ are equivalences in
	$\ct{Fib}(\ct{B})$. 
\end{theorem}
\begin{proof}
	Thanks to \prx\ref{propcatoverbhasdescentobj} and
	\prx\ref{propmonadpreservescodescentobj}, \citet{Lack:codoc} (see
	\thx\ref{thm:cohcom}) guarantees that the forgetful 2-functor from strict algebras to
	pseudo-algebras for $\M$ has a left 2-adjoint, such that the components of the unit are
	equivalences between pseudo-algebras. Dually, the forgetful 2-functor from strict
	coalgebras to pseudo-coalgebras for $\N$ has a right 2-adjoint, such that the components
	of the counit are equivalences between pseudo-coalgebras. Both forgetful 2-functors
	coincide (up to isomorphism) with $V$, by \thx\ref{cps} and \thx\ref{mthm}. 
\end{proof}

Since adjoints are unique up to isomorphism, the two splitting adjoints of
\thx\ref{theortwosplittings} need to be realized by the explicit constructions of
\citet{GiraudJ:cohna}. But we can do more than this. \citet{Lack:codoc} also gives
a concrete recipe to calculate the strictification of pseudo-algebras, in terms of
codescent objects. Dually, for the strictification of pseudo-coalgebras in terms of
descent objects. In the rest of this section, we show that our conceptual proof of
\thx\ref{theortwosplittings} also produces the explicit splitting constructions of
\citet{GiraudJ:cohna}. 

\begin{remark}\label{remrightsplitting}
	{Using the explicit construction of descent objects in
		$\ct{Cat}/\ct{B}$ given in \prx\ref{propcatoverbhasdescentobj}, one can recover the
		explicit construction of the right splitting of fibrations of \citet{GiraudJ:cohna}.}
	Let $p:\ct{E}\futo\ct{B}$ be a cloven fibration, with pseudo-coalgebra structure for the
	2-comonad $\N$ given by a map $\alpha:p\to N(p)$ and isomorphisms $\alpha_{\eta'}$ and
	$\alpha_{\mu'}$ (see \rmx\ref{remcoalgstructoffib}). Following the concrete recipe of
	\citet{Lack:codoc} to strictify pseudo-coalgebras, we have that $R(p)$ is the
	descent object in $\ct{SpFib}(\ct{B})$ of the coherence data 
	$$\begin{tikzcd}
		N(p) \arrow[r,shift left=3ex,"{\mu'_p}"]\arrow[r,shift right=2.5ex,"{N\alpha}"']& N^2(p) \arrow[l,shift left=0.5ex,"{N\eta'_p}"'{inner sep=0.15ex}]\arrow[r,shift left=3ex,"{\mu'_{Np}}"]\arrow[r,shift right=2.5ex,"{N^2\alpha}"'] \arrow[r,shift right=0.5ex,"{N\mu'_p}"{inner sep=0.15ex}]& N^3(p)
	\end{tikzcd}$$
	By \prx\ref{propcatoverbhasdescentobj} and \prx\ref{propmonadpreservescodescentobj}, such descent object is calculated in $\ct{Cat}/\ct{B}$, as an inserter followed by two equifiers. The first step is to calculate the iso-inserter of
	$$\begin{tikzcd}
		N(p) \arrow[r,shift left=0.8ex,"{\mu'_p}"]\arrow[r,shift right=0.8ex,"{N\alpha}"']& N^2(p)
	\end{tikzcd}$$
	By \prx\ref{propcatoverbhasdescentobj}, such iso-inserter is given by the category whose
	objects are all pairs $({\ple{b,X}},\phi)$ with ${\ple{b,X}}$ in $\ct{G}_p$ and
	$\phi:\mu'_p(\ple{b,X})\cong (N\alpha)(\ple{b,X})$ an isomorphism in $\ct{G}_{Np}$ over
	the identity, and whose morphisms $({\ple{b,X}},\phi)\to ({\ple{b',X'}},\phi')$ are
	morphisms $\ple{h,\lambda}$ in $\ct{G}_p$ such that $(N\alpha)(\ple{h,\lambda})\circ
	\phi=\phi'\circ \mu'_p(\ple{h,\lambda})$. 
	
	In the notation of \rmx\ref{remdescr}, $\phi$ is a natural isomorphism
	$$\begin{tikzcd}
		\ct{B}/b\ar{rr}{\overline{X}}[name=a,']{}\ar{rd}[']{X}&&{\ct{G}_p}\\
		&\ct{E}\ar{ru}[name=b]{}[']{\alpha}&
		\ar["\scalebox{.3}{$\bullet$}"{sloped,anchor=south},from=a,to=b]{}[']{\phi}
	\end{tikzcd}$$
	over $\ct{B}$. For every $f:a\to b$ in $\ct{B}$, the component $\phi_f$ of $\phi$ on $f$ is given by a natural isomorphism 
	$$\begin{tikzcd}
		&\ct{B}/b\ar{rd}[name=b,']{}[]{X}\\
		\ct{B}/a\ar{rr}[']{(-)^*X(f)}[name=a]{}\ar{ru}[{inner sep=0.2ex}]{f\circ -}&&\ct{E}
		\ar["\scalebox{.3}{$\bullet$}"'{sloped,anchor=south},from=b,to=a]{}[']{\phi_f}
	\end{tikzcd}$$
	over $\ct{B}$. For every $u:a'\to a$ in $\ct{B}$, the component $\phi_{f,u}$ of $\phi_f$ on $u$ is given by an isomorphism
	$$\phi_{f,u}:X(fu)\cong u^*X(f)$$
	in $\ct{E}$ over $\id{a'}$. Naturality of $\phi_f$ on $u$ yields in particular a commutative triangle
	$$\begin{tikzcd}[row sep=3.5ex,column sep=3.5ex]
		X(fu)\ar{rr}{X(u)}[name=a,']{}\ar{rd}[']{\phi_{f,u}}&&X(f)\\
		&u^*X(f)\ar{ru}[name=b]{}[']{\opr{Cart}(u,X(f))}
	\end{tikzcd}$$
	where $\opr{Cart}(u,X(f))$ is the chosen cartesian lifting of $u$ to $X(f)$. This triangle
	subsumes all the other conditions $\phi$ needs to satisfy, by cartesianity arguments. 
	
	So an arbitrary object of the iso-inserter is equivalently given by $\ple{b,X}$ equipped
	with chosen isomorphisms $\phi_{f,u}$ that satisfy the commutative triangle above. Under
	the axiom of choice, this is equivalent to restrict to those $\ple{b,X}$ in $\ct{G}_p$
	such that $X$ is a cartesian functor, \ie a functor that preserves the cleavage up to
	isomorphism.
	
	The naturality-like condition that the morphisms of the iso-inserter need to satisfy is
	equally subsumed under the commutative triangle above. So that morphisms in the
	iso-inserter are simply all morphisms $\ple{h,\lambda}$ in $\ct{G}_p$. 
	
	In order to produce a descent object, we should then take two equifiers to force the following two conditions (the second one is a cocycle condition):
	$$\phi_{f,\id{}}=\opr{\alpha_{\eta'}}:X(f)\cong(\id{})^*X(f)$$
	$$\begin{tikzcd}
		X(fuv)\arrow[r,"{\phi_{f,uv}}"]\arrow[d,"{\phi_{fu,v}}"'] & (uv)^*X(f) \arrow[d,iso,"{\alpha_{\mu'}}"{inner sep=1.15ex}]\\
		v^*X(fu)\arrow[r,"{v^*\phi_{f,u}}"'] & v^*u^*X(f)
	\end{tikzcd}$$
	{where the isomorphisms $\alpha_{\eta'}$ and $\alpha_{\mu'}$ are those of the
		pseudo-coalgebra structure $\alpha$ on $p$, \ie those given by the chosen cleavage of
		$p$ (see \rmx\ref{remcoalgstructoffib}). Both conditions are however true for all
		$(\ple{b,X},\phi)$, as they are subsumed by the commutative triangle above,
		by cartesianity arguments. So the iso-inserter calculated above is already the descent
		object we needed.}
	
	We conclude that the concrete recipe of \citet{Lack:codoc} to strictify pseudo-coalgebras 
	translates as taking $R(p)$ to be the restriction of $N(p)$ to those objects $\ple{b,X}$
	with $X$ that preserves the cleavage up to isomorphism (actually with chosen
	isomorphisms), without any further condition on morphisms. This is precisely
	the right splitting construction of \citet{GiraudJ:cohna}.
\end{remark}

\begin{remark}\label{remessimgivesrightsplitting}
	We note another way to recover the right adjoint splitting of
	fibrations. By \thx\ref{mthm}, the pseudo-coalgebra structure $\alpha$ of a cloven
	fibration $p$ is a right adjoint right inverse of the counit $\eta'_F:N(F)\futo F$. This
	$\alpha$ is a homomorphism of 
	fibrations which embeds $F$ into the split fibration $N(F)$. Closing under isomorphism the
	image of $F$ into $N(F)$ provides a split fibration equivalent to $F$.
	
	This recovers also the right splitting construction of \citet{GiraudJ:cohna}, as by
	\rmx\ref{remcoalgstructoffib} such image is made of triangles 
	$$\begin{tikzcd}[row sep=4.75ex,column sep=4.75ex]
		\ct{B}/{p(E)}\ar{rr}{(-)^* E}\ar{rd}[']{\partial_0}&&\ct{E}\ar{ld}{p}\\
		&\ct{B}
	\end{tikzcd}$$
	and $(-)^*E$ preserves the cleavage up to isomorphism.
\end{remark}

In the following \rmx\ref{remleftsplitting}, we recover the explicit construction of the
left splitting of fibrations given in \citep{GiraudJ:cohna}. Notice that the
explicit construction of codescent objects in $\ct{Cat}$ is much more convoluted than the
one of descent objects. So it is better to use the concrete recipe of strictification of
pseudo-coalgebras given by \citet{Power:gencr}, later refined by \citet{Lack:codoc}. Such
an approach uses enhanced factorization systems. We recall
Lack's result \citeyearpar[Theorem~4.10]{Lack:codoc}, built on results of
\citet{Power:gencr}.

\begin{theorem}[Lack, Power]\label{thm:cohmnd}
	If $\ct{K}$ is a 2-category with an enhanced factorization system $(\ct{E}, \ct{M})$
	having the property that if $j$ in \ct{M} and $jk \cong 1$ then $kj \cong 1$, and if $T$
	is a 2-monad on $\ct{K}$ for which $Tf$ in \ct{E} whenever $f$ in \ct{E}, then the
	inclusion $T\text{-Alg}_s \futo \text{Ps-}T\text{-Alg}$ has a left adjoint, and the
	components of the unit of the adjunction are equivalences in $\text{Ps-}T\text{-Alg}$.
\end{theorem}

While it seems that \rmx\ref{remessimgivesrightsplitting} could as well be inscribed in
this idea, it is hard to capture the essential image in a strict factorization system. 

We now use Lack (and Power)'s concrete proof of \thx\ref{thm:cohmnd} to recover Giraud's
explicit construction of the left splitting of fibrations. 

\begin{remark}\label{remleftsplitting}
	Recall that a factorization system on a 2-category is \dfn{enhanced} when squares
	which are pseudo morphisms in $\ct{Cat}^{\ct{2}}$ with domain in \ct{E} and codomain in
	\ct{M}---so they commute up to an invertible 2-cell $\alpha$---can be filled with a unique
	diagonal such that the upper triangle commutes strictly and the lower one commutes up to a
	unique invertible 2-cell which coincides with $\alpha$ on the whole square.
	
	{The factorization system on $\ct{Cat}$ where \ct{E} consists of the
		bijective-on-objects functors and 
		\ct{M} consists of the fully-faithful functors, is enhanced. This lifts 
		to a factorization system on the 2-category $\ct{Cat}/\ct{B}$, which is easily seen to be
		enhanced.}
	
	The other two assumptions of \thx\ref{thm:cohmnd} are easily verified for the 2-monad
	$\M$. Indeed a fully faithful functor with a right quasi-inverse is clearly an
	equivalence. And it is easily seen from the definition of the 2-monad $\M$ (see
	\prx\ref{propdefmonad}) that $M$ preserves bijective-on-objects functors. In fact, the
	action of $M$ on morphisms in $\ct{Cat}/\ct{B}$, induced by the universal property of the
	comma object, is almost trivial. 
	
	Power's insight is that an enhanced factorization system can be used to extract a strict
	algebra from a pseudo one. In particular, the underlying object of the strict algebra can
	be computed factoring the pseudo-algebra map. 
	
	Let $p:\ct{E}\futo\ct{B}$ be a cloven fibration, with pseudo-algebra structure for the
	2-monad $\M$ given by a map $\alpha:M(p)\to p$ and isomorphisms $\alpha_\eta$ and
	$\alpha_\mu$ that regulate respectively liftings in the chosen cleavage of $p$ of an
	identity and of a composite. Using (the proof of) \thx\ref{thm:cohmnd}, we obtain that the
	underlying object of the left splitting $L(p)$ of $p$ is the full image of the
	pseudo-algebra map $\alpha:M(p)\to p$. That is, the objects are the same of $M(p)$, but an
	arrow from $(f,a)$ to $(f',a')$ is an arrow $\alpha(f,a) \to \alpha(f',a')$ in $\ct{E}$
	and thus an arrow $f^*a\to (f')^*a'$ in $\ct{E}$. This is precisely the left splitting
	construction of \citet{GiraudJ:cohna}.
\end{remark}

\bibliographystyle{plainnat}
\bibliography{MRS_biblio}

\end{document}